\documentclass[12pt,a4paper]{amsart}
\allowdisplaybreaks[1]

\usepackage[top=30truemm,bottom=30truemm,left=25truemm,right=25truemm]{geometry}

\DeclareFontEncoding{OT2}{}{}
\DeclareFontSubstitution{OT2}{cmr}{m}{l}

\DeclareFontFamily{OT2}{cmr}{\hyphenchar\font45 }
\DeclareFontShape{OT2}{cmr}{m}{l}{%
<5><6><7><8><9>gen*wncyr%
<10><10.95><12><14.4><17.28><20.74><24.88>wncyr10}{}

\DeclareMathAlphabet{\mathcyr}{OT2}{cmr}{m}{l}
\DeclareMathAlphabet{\mathcyb}{OT2}{cmr}{b}{l}
\SetMathAlphabet{\mathcyr}{bold}{OT2}{cmr}{b}{l}

\usepackage{amstext}
\usepackage{amsthm}
\usepackage{amssymb}

\newtheorem{thm}{Theorem}[section]
\newtheorem{lem}[thm]{Lemma}
\newtheorem{prop}[thm]{Proposition}
\newtheorem{cor}[thm]{Corollary}

\theoremstyle{definition}
\newtheorem{defn}[thm]{Definition}

\theoremstyle{remark}
\newtheorem{rem}[thm]{Remark}

\newcommand{\sha}{\mathbin{\widetilde{\mathcyr{sh}}}}
\newcommand{\sh}{\mathbin{\mathcyr{sh}}}

\begin{document}

\title[Generating functions for Ohno type sums of FMZSVs and SMZSVs]{Generating functions for Ohno type sums of finite and symmetric multiple zeta-star values}

\author{Minoru Hirose}
\address[Minoru Hirose]{Faculty of Mathematics, Kyushu University
 744, Motooka, Nishi-ku, Fukuoka, 819-0395, Japan}
\email{m-hirose@math.kyushu-u.ac.jp}

\author{Hideki Murahara}
\address[Hideki Murahara]{Nakamura Gakuen University Graduate School,
 5-7-1, Befu, Jonan-ku, Fukuoka, 814-0198, Japan} 
\email{hmurahara@nakamura-u.ac.jp}

\author{Shingo Saito}
\address[Shingo Saito]{Faculty of Arts and Science, Kyushu University,
 744, Motooka, Nishi-ku, Fukuoka, 819-0395, Japan}
\email{ssaito@artsci.kyushu-u.ac.jp}


\begin{abstract}
 Ohno's relation states that a certain sum, which we call an Ohno type sum, of multiple zeta values remains unchanged 
 if we replace the base index by its dual index.
 In view of Oyama's theorem concerning Ohno type sums of finite and symmetric multiple zeta values,
 Kaneko looked at Ohno type sums of finite and symmetric multiple zeta-star values and made a conjecture 
 on the generating function for a specific index of depth three.
 In this paper, we confirm this conjecture and further give a formula for arbitrary indices of depth three.
\end{abstract}

\maketitle

%%%%%%%%%%%%%%%%%%%%%%%%%%%%%%%%%%%%%%%%%%%%%%%%%%%%%%%%%%%%%%%%%%%%%%%%%%%%%%%%%%%%%%%%%%%%%%%%%%%%
\section{Introduction}
\subsection{Finite and symmetric multiple zeta(-star) values}
For positive integers $k_1,\dots,k_r$ with $k_r \ge2$, the multiple zeta values and the multiple zeta-star values are defined by 
\begin{align*}
 \zeta(k_1,\dots, k_r)
 &=\sum_{1\le n_1<\cdots <n_r} \frac{1}{n_1^{k_1}\cdots n_r^{k_r}} \in \mathbb{R}, \\
 \zeta^\star (k_1,\dots, k_r)
 &=\sum_{1\le n_1\le \cdots \le n_r} \frac{1}{n_1^{k_1}\cdots n_r^{k_r}} \in \mathbb{R}. 
\end{align*} 

Kaneko and Zagier \cite{KZ19} introduced the finite multiple zeta(-star) values and
the symmetric multiple zeta(-star) values. 
Set $\mathcal{A}:=\prod_p\mathbb{F}_p/\bigoplus_p\mathbb{F}_p$, where $p$ runs over all primes.
For positive integers $k_1,\dots,k_r$, the finite multiple zeta(-star) values are defined by
\begin{align*}
 \zeta_\mathcal{A}(k_1,\dots,k_r)
 &=\Biggl(\sum_{1\le m_1<\dots<m_r<p}\frac{1}{m_1^{k_1}\dotsm m_r^{k_r}}\bmod p\Biggr)_p\in\mathcal{A},\\
 \zeta_\mathcal{A}^{\star}(k_1,\dots,k_r)
 &=\Biggl(\sum_{1\le m_1\le\dots\le m_r<p}\frac{1}{m_1^{k_1}\dotsm m_r^{k_r}}\bmod p\Biggr)_p\in\mathcal{A}.
\end{align*}
Let $\mathcal{Z}$ be the $\mathbb{Q}$-linear subspace of $\mathbb{R}$ spanned by $1$ and all multiple zeta values. 
For positive integers $k_1,\dots,k_r$, we define the symmetric multiple zeta(-star) values by
\begin{align*}
 \zeta_\mathcal{S}(k_1,\dots,k_r)
 &=\sum_{i=0}^{r}(-1)^{k_{i+1}+\dots+k_r}\zeta(k_1,\dots,k_i)
 \zeta(k_r,\dots,k_{i+1})\bmod\zeta(2)\in\mathcal{Z}/\zeta(2)\mathcal{Z},\\
 \zeta_\mathcal{S}^{\star}(k_1,\dots,k_r)
 &=\sum_{i=0}^{r}(-1)^{k_{i+1}+\dots+k_r}\zeta^{\star}(k_1,\dots,k_i)\zeta^{\star}(k_r,\dots,k_{i+1})\bmod\zeta(2)\in\mathcal{Z}/\zeta(2)\mathcal{Z},
\end{align*}
where we understand $\zeta(\emptyset)=\zeta^{\star}(\emptyset)=1$.
Here, the symbols $\zeta$ and $\zeta^\star$ on the right-hand side mean the regularized values coming from harmonic regularization, 
i.e.\ real values obtained by taking constant terms of harmonic regularization 
as explained in \cite{IKZ06}. 

Denoting by $\mathcal{Z}_{\mathcal{A}}$ the $\mathbb{Q}$-vector subspace of $\mathcal{A}$ spanned by $1$ and all finite multiple zeta values, Kaneko and Zagier conjectured that there is an isomorphism between $\mathcal {Z}_{\mathcal{A}}$ and $\mathcal{Z}/\zeta(2)\mathcal{Z}$ as $\mathbb{Q}$-algebras such that $\zeta_{\mathcal{A}}(k_1,\ldots, k_r)$ and $\zeta_{\mathcal{S}}(k_1,\ldots, k_r)$ correspond to each other 
(for details, see \cite{Kan19} and \cite{KZ19}). 
In the sequel, the letter $\mathcal{F}$ stands for either $\mathcal{A}$ or $\mathcal{S}$.

%%%%%%%%%%%%%%%%%%%%%%%%%%%%%%%%%%%%%%%%%%%%%%%%%%%%%%%%%%%%%%%%%%%%%%%%%%%%%%%%%%%%%%%%%%%%%%%%%%%%
\subsection{Main results}
For a sequence $\boldsymbol{k}=(k_1,\dots,k_r)$, we call $\left|\boldsymbol{k}\right|=k_1+\cdots+k_r$ its weight and $r$ its depth.  
For two sequences $\boldsymbol{k}$ and $\boldsymbol{l}$ of the same depth, we denote by $\boldsymbol{k} \oplus \boldsymbol{l}$ the sequence obtained by componentwise addition. 
We call a (possibly empty) sequence of positive integers an index. 
%An index is said to be admissible if either it is empty or its last component is greater than $1$.
%
%We denote by $\mathcal{I}$ the $\mathbb{Q}$-linear space spanned by the indices.
Throughout this paper, we always assume that $\boldsymbol{e}$ runs over sequences of nonnegative integers having suitable depth. 

Ohno obtained the following remarkable result: 
\begin{thm}[Ohno's relation; Ohno \cite{Ohn99}] \label{ohno}
 For a nonempty index $\boldsymbol{k}$ whose last component is greater than $1$ and a nonnegative integer $m$, we have
 \begin{align*}
  \sum_{\left|\boldsymbol{e}\right|=m}
  \zeta (\boldsymbol{k}\oplus\boldsymbol{e}) 
  =\sum_{\left|\boldsymbol{e}\right|=m}
  \zeta (\boldsymbol{k}^\dagger \oplus\boldsymbol{e}),
 \end{align*}
 where $\boldsymbol{k}^\dagger$ is the dual index of $\boldsymbol{k}$ (see \cite{Ohn99} for a precise definition). 
\end{thm}
\begin{defn}[Hoffman's dual index]
 For a nonempty index $\boldsymbol{k}=(k_1,\dots,k_r)$, we define Hoffman's dual index of $\boldsymbol{k}$ by
 \begin{align*}
  \boldsymbol{k}^{\vee}=(\underbrace{1,\dots,1}_{k_1}+\underbrace{1,\dots,1}_{k_2}+1,\dots,1+\underbrace{1,\dots,1}_{k_r}).
 \end{align*}
\end{defn}
In contrast to Theorem \ref{ohno}, Oyama proved the following:
\begin{thm}[Oyama \cite{Oya15}] \label{ohnoF}
 For a nonempty index $\boldsymbol{k}$ and  a nonnegative integer $m$, we have
 \[ 
  \sum_{\left|\boldsymbol{e}\right|=m}
  \zeta_\mathcal{F} (\boldsymbol{k}\oplus\boldsymbol{e})
  =\sum_{\left|\boldsymbol{e}\right|=m}
  \zeta_\mathcal{F} ((\boldsymbol{k}^\vee \oplus\boldsymbol{e})^\vee ).
 \]
\end{thm}
For a positive integer $k$, let 
\[
 \mathfrak{Z}_{\mathcal{A}}(k)=( B_{p-k}/k \bmod{p} )_p\in\mathcal{A}, \qquad 
 \mathfrak{Z}_{\mathcal{S}}(k)=\zeta(k) \bmod{\zeta(2)}\in\mathcal{Z}/\zeta(2)\mathcal{Z},
\]
where $B_n$ is the $n$-th Bernoulli number. 
In light of this theorem, Kaneko looked at the generating function 
\[
 \mathit{O}_\mathcal{F} (\boldsymbol{k})
 =\sum_{\boldsymbol{e}} 
 \zeta_\mathcal{F}^\star (\boldsymbol{k} \oplus \boldsymbol{e}) X^{\left|\boldsymbol{k} \oplus \boldsymbol{e}\right|}
 +\sum_{\boldsymbol{e}}
 (-1)^{\left|\boldsymbol{e}\right|} \zeta_\mathcal{F}^\star (\boldsymbol{k}^\vee \oplus \boldsymbol{e}) X^{\left|\boldsymbol{k}^\vee \oplus \boldsymbol{e}\right|} 
\]
and gave the following conjecture:
\[
 \mathit{O}_\mathcal{F} (2,1,2)
 =( 3\mathfrak{Z}_\mathcal{F}(3)X^3 + 5\mathfrak{Z}_\mathcal{F}(5)X^5  + 7\mathfrak{Z}_\mathcal{F}(7)X^7 +\cdots)^{2}.
\]
In this paper, we prove a theorem that generalizes this conjecture. 
For positive integers $k$ and $i$, let
\begin{align*}
 F_{k,i}(X)
 =\sum_{n\ge k+i} \left( (-1)^k \binom{n-1}{k-1} -(-1)^i \binom{n-1}{i-1} \right)
  \mathfrak{Z}_{\mathcal{F}}(n)X^n.
\end{align*}
%%%%%
\begin{thm}[Main theorem] \label{main}
 For positive integers $k_1,k_2,k_3$, we have
 \begin{align*}
  \mathit{O}_\mathcal{F} (k_1,k_2,k_3) 
  =
  \begin{cases}
   \displaystyle{ 
    -\sum_{\substack{i+j=k_2+1 \\ i,j\ge2}} 
    F_{k_1,i}(X) F_{k_3,j}(X) \qquad (k_2\ge2), 
   } \\ 
    F_{k_1,1}(X) F_{k_3,1}(X) \qquad\qquad\qquad\, (k_2=1).
  \end{cases}
 \end{align*}
\end{thm}
\begin{rem}
 Theorem \ref{main} implies that $\mathit{O}_\mathcal{F} (k_1,2,k_3)=0$. 
 We also note that the theorem implies Kaneko's conjecture when $k_1=2,k_2=1,k_3=2$, 
 since $F_{2,1}(X)=\sum_{n\ge3} n\mathfrak{Z}_\mathcal{F}(n)X^n$ and $\mathfrak{Z}_\mathcal{F}(2n)=0$ for all positive integers $n$. 
\end{rem}
\begin{rem}
 By the duality formula (Theorem \ref{dualFS}), the second sum in $\mathit{O}_\mathcal{F} (\boldsymbol{k})$ is equal to 
 $-\sum_{\boldsymbol{e}}
  (-1)^{\left|\boldsymbol{e}\right|} 
  \zeta_\mathcal{F}^\star ((\boldsymbol{k}^\vee \oplus \boldsymbol{e})^\vee) 
  X^{\left|\boldsymbol{k}^\vee \oplus \boldsymbol{e}\right|}$, 
  which more closely resembles the right-hand side of Theorem \ref{ohnoF}.   
\end{rem}
\begin{rem}
 We note that $\mathit{O}_\mathcal{F} (k)=0$ and $\mathit{O}_\mathcal{F} (k_1,k_2)=0$ hold for all positive integers $k,k_1,k_2$ (see Section 2).
\end{rem}
\begin{rem}
 Hirose-Imatomi-Murahara-Saito \cite{HIMS18} shows that 
 $\sum_{\left|\boldsymbol{e}\right|=m} 
 \zeta_\mathcal{F}^\star (\boldsymbol{k}^\vee \oplus \boldsymbol{e})$ 
 can be represented as a $\mathbb{Z}$-linear combination of $\zeta_\mathcal{F}^\star (\boldsymbol{k} \oplus \boldsymbol{e})$'s.
\end{rem}

We will give proofs of $\mathit{O}_\mathcal{F} (k)=0$ and $\mathit{O}_\mathcal{F} (k_1,k_2)=0$ in Section 2 
and of our main theorem (Theorem \ref{main}) in Section 3.

%%%%%%%%%%%%%%%%%%%%%%%%%%%%%%%%%%%%%%%%%%%%%%%%%%%%%%%%%%%%%%%%%%%%%%%%%%%%%%%%%%%%%%%%%%%%%%%%%%%%
\section{Proofs of $\mathit{O}_\mathcal{F} (k)=0$ and $\mathit{O}_\mathcal{F} (k_1,k_2)=0$}
For an index $\boldsymbol{k}$ and a positive integer $k$, we let 
\[
 \tilde\zeta_\mathcal{F} (\boldsymbol{k})
  =\zeta_\mathcal{F} (\boldsymbol{k}) X^{\left|\boldsymbol{k}\right|}, \quad 
 \tilde\zeta_\mathcal{F}^\star (\boldsymbol{k})
  =\zeta_\mathcal{F}^\star (\boldsymbol{k}) X^{\left|\boldsymbol{k}\right|}, \text{ and} \quad 
 \tilde{\mathfrak{Z}}_{\mathcal{F}} (k) 
  =\mathfrak{Z}_\mathcal{F}(k) X^{k}.
\]
%%%%%%%%
\subsection{Proof of $\mathit{O}_\mathcal{F} (k)=0$}
\begin{prop} \label{dep1}
 For a positive integer $k$, we have
 \[
  \tilde\zeta_\mathcal{F} (k)
  =\tilde\zeta_\mathcal{F}^\star (k)
  =0. 
 \]
\end{prop}
\begin{proof}
 See Kaneko \cite{Kan19}. 
\end{proof}
\begin{prop}[Hoffman \cite{Hof15}, Murahara \cite{Mur15}] \label{symsumF}
 For a nonempty index $(k_1,\dots,k_r)$, we have
 \begin{align*}
  \sum_{\sigma \in S_r} \tilde\zeta_\mathcal{F} (k_{\sigma(1)},\dots, k_{\sigma(r)})
  =\sum_{\sigma \in S_r} \tilde\zeta_\mathcal{F}^\star (k_{\sigma(1)},\dots, k_{\sigma(r)})
  =0,
 \end{align*}
  where $S_r$ is the symmetric group of degree $r$.
\end{prop}
For a nonnegative integer $m$, we let $\{1\}^m$ denote the all-one sequence of length $m$. 
\begin{prop} \label{main_dep1}
 For a positive integer $k$, we have
 \[
  \mathit{O}_\mathcal{F} (k)=0. 
 \]
\end{prop}
\begin{proof}
 By Propositions \ref{dep1} and \ref{symsumF}, we have 
 \begin{align*}
  \mathit{O}_\mathcal{F} (k) 
  &=\sum_{m\ge0} 
     \tilde\zeta_\mathcal{F}^\star (k+m) 
  +\sum_{\boldsymbol{e}}
    (-1)^{\left|\boldsymbol{e}\right|} \tilde\zeta_\mathcal{F}^\star ((\{1\}^k) \oplus \boldsymbol{e}) \\
  &=0+0=0. \qedhere
 \end{align*}
\end{proof}

%%%%%%%%%%%%%%%%%%%%%%%%
\subsection{Proof of $\mathit{O}_\mathcal{F} (k_1,k_2)=0$}
\begin{prop} \label{dep2}
 For positive integers $k_1,k_2$, we have
 \[
  \tilde\zeta_\mathcal{F} (k_1,k_2)
  =\tilde\zeta_\mathcal{F}^\star (k_1,k_2)
  =(-1)^{k_1+1} \binom{k_1+k_2}{k_1} \tilde{\mathfrak{Z}}_{\mathcal{F}} (k_1+k_2). 
 \]
\end{prop}
\begin{proof}
 See Kaneko \cite{Kan19}. 
\end{proof}
The next formula is well known (see \cite{SS17}, for example).
\begin{prop} \label{antipode}
 For a nonempty index $(k_1,\dots,k_r)$, we have
 \[
  \sum_{i=0}^{r}(-1)^{i} 
   \tilde\zeta_{\mathcal{\mathcal{F}}}^{\star} (k_1,\dots,k_{i})
   \tilde\zeta_\mathcal{F} (k_{r,},\dots,k_{i+1})=0.
 \]
 Here, we understand $\tilde\zeta_\mathcal{F}(\emptyset)=\tilde\zeta_{\mathcal{F}}^{\star}(\emptyset)=1$.
\end{prop}
\begin{thm}[Duality formula; Hoffman \cite{Hof15}, Jarossay \cite{Jar14}] \label{dualFS}
 For a nonempty index $\boldsymbol{k}$, we have
 \[ \tilde\zeta_\mathcal{F}^\star (\boldsymbol{k})=-\tilde\zeta_\mathcal{F}^\star (\boldsymbol{k}^{\vee}). \]
\end{thm}
\begin{prop} \label{main_dep2}
 For positive integers $k_1,k_2$, we have
 \[
  \mathit{O}_\mathcal{F} (k_1,k_2)=0. 
 \]
\end{prop}
\begin{proof}
 By Propositions \ref{symsumF} and \ref{antipode}, we have 
 \begin{align*} 
  &\sum_{\boldsymbol{e}}
    (-1)^{\left|\boldsymbol{e}\right|} \tilde\zeta_\mathcal{F}^\star ((\{1\}^{k_1-1},2,\{1\}^{k_2-1}) \oplus \boldsymbol{e}) \\
  &=\sum_{\boldsymbol{e}}
    (-1)^{k_1+k_2+\left|\boldsymbol{e}\right|} \tilde\zeta_\mathcal{F} ((\{1\}^{k_2-1},2,\{1\}^{k_1-1}) \oplus \boldsymbol{e}).
 \end{align*}
 By Theorem \ref{ohnoF}, we have
 \begin{align*} 
  &\sum_{\boldsymbol{e}}
    (-1)^{k_1+k_2+\left|\boldsymbol{e}\right|} \tilde\zeta_\mathcal{F} ((\{1\}^{k_2-1},2,\{1\}^{k_1-1}) \oplus \boldsymbol{e}) \\
  &=\sum_{\boldsymbol{e}}
    (-1)^{k_1+k_2+\left|\boldsymbol{e}\right|} \tilde\zeta_\mathcal{F} (((k_2,k_1) \oplus \boldsymbol{e})^\vee).
 \end{align*}
 By Propositions \ref{symsumF}, \ref{antipode}, and Theorem \ref{dualFS}, we have
 \begin{align*} 
  &\sum_{\boldsymbol{e}}
    (-1)^{k_1+k_2+\left|\boldsymbol{e}\right|} \tilde\zeta_\mathcal{F} (((k_2,k_1) \oplus \boldsymbol{e})^\vee) \\
  &=\sum_{\boldsymbol{e}}
    \tilde\zeta_\mathcal{F}^\star (((k_1,k_2) \oplus \boldsymbol{e})^\vee) \\
  &=-\sum_{\boldsymbol{e}}
    \tilde\zeta_\mathcal{F}^\star ((k_1,k_2) \oplus \boldsymbol{e}). 
 \end{align*}
 Then we have    
 \begin{align*}
  &\mathit{O}_\mathcal{F} (k_1,k_2) \\
  &=\sum_{\boldsymbol{e}}
    \tilde\zeta_\mathcal{F}^\star ((k_1,k_2) \oplus \boldsymbol{e})  
  +\sum_{\boldsymbol{e}}
    (-1)^{\left|\boldsymbol{e}\right|} \tilde\zeta_\mathcal{F}^\star ((\{1\}^{k_1-1},2,\{1\}^{k_2-1}) \oplus \boldsymbol{e}) \\
  &=0. \qedhere
 \end{align*}
\end{proof}

%%%%%%%%%%%%%%%%%%%%%%%%%%%%%%%%%%%%%%%%%%%%%%%%%%%%%%%%%%%%%%%%%%%%%%%%%%%%%%%%%%%%%%%%%%%%%%%%%%%%
\section{Proof of Theorem \ref{main}}
%%%%%%%%%%%%%%%%%%%%%%%%
\subsection{Properties of $\tilde\zeta_\mathcal{F}$ and $\tilde\zeta_\mathcal{F}^\star$}
To prove our main theorem (Theorem \ref{main}), we need Lemmas \ref{lemA}, \ref{lemB}, and \ref{PQ2}. 
The following known results will be used to prove these lemmas. 
\begin{prop}[Reversal formula] \label{revF}
 For an index $(k_1,\dots,k_r)$, we have
 \begin{align*}
   \zeta_\mathcal{F} (k_1,\dots,k_r)
   &=(-1)^{k_1+\cdots+k_r} \zeta_\mathcal{F} (k_r,\dots,k_1). 
 \end{align*}
\end{prop}
\begin{prop} \label{112}
 For nonnegative integers $a$ and $b$, we have
 \begin{align*}
  \tilde\zeta_\mathcal{F} (\{1\}^{a},2,\{1\}^{b}) 
  =\tilde\zeta_\mathcal{F}^{\star} (\{1\}^{a},2,\{1\}^{b}) 
  =(-1)^{a+1} \binom{a+b+2}{a+1} \tilde{\mathfrak{Z}}_\mathcal{F}(a+b+2).
 \end{align*}
\end{prop}
\begin{proof}
 By Propositions \ref{symsumF}, \ref{antipode}, and \ref{revF}, we have
 \begin{align*}
  \tilde\zeta_\mathcal{F} (\{1\}^{a},2,\{1\}^{b}) 
  &=\tilde\zeta_\mathcal{F}^{\star} (\{1\}^{a},2,\{1\}^{b}).
 \end{align*} 
 By Proposition \ref{dep2} and Theorem \ref{dualFS}, we have
 \begin{align*}
  \tilde\zeta_\mathcal{F}^{\star} (\{1\}^{a},2,\{1\}^{b}) 
  &=-\tilde\zeta_\mathcal{F}^{\star} (a+1,b+1) \\
  &=(-1)^{a+1} \binom{a+b+2}{a+1} \tilde{\mathfrak{Z}}_\mathcal{F}(a+b+2).
 \end{align*} 
 This finishes the proof. 
\end{proof}
\begin{cor} \label{112no2}
 For nonnegative integers $a$ and $b$, we have
 \begin{align*}
  \tilde\zeta_\mathcal{F}^{\star} (\{1\}^{a},2,\{1\}^{b}) 
  =(-1)^{a+b+1} \tilde\zeta_\mathcal{F}^{\star} (\{1\}^{a},2,\{1\}^{b}).
 \end{align*}
\end{cor}
\begin{thm}[Sum formula; Saito-Wakabayashi \cite{SW15}, Murahara \cite{Mur15}] \label{sumF}
 For nonnegative integers $i$ and $j$, we have
 \begin{align*}
  \sum_{\substack{ k_1,\dots,k_{i+j+1}\ge1 \\ k_{i+1}\ge2 }} \tilde\zeta_\mathcal{F} (k_1,\dots,k_{i+j+1}) 
  &=F_{i+1,j+1}(X), \\
  \sum_{\substack{ k_1,\dots,k_{i+j+1}\ge1 \\ k_{i+1}\ge2 }} \tilde\zeta_\mathcal{F}^\star (k_1,\dots,k_{i+j+1})   
  &=(-1)^{i+j+1} F_{i+1,j+1}(X). 
 \end{align*}
\end{thm}
We denote by $\mathcal{I}$ the space of formal $\mathbb{Q}$-linear combinations of indices. 
We define a $\mathbb{Q}$-bilinear product $\sha$ on $\mathcal{I}$ inductively by setting
\begin{align*}
 \boldsymbol{k} \sha \emptyset 
 &=\emptyset \sha \boldsymbol{k}=\boldsymbol{k}, \\ 
 (k _1, \boldsymbol{k}) \sha (l_1,\boldsymbol{l}) 
 &=(k_1, \boldsymbol{k} \sha (l_1,\boldsymbol{l})) + (l_1, (k_1,\boldsymbol{k}) \sha \boldsymbol{l})
\end{align*} 
for all indices $\boldsymbol{k}$, $\boldsymbol{l}$ and all positive integers $k_1$, $l_1$. 
We $\mathbb{Q}$-linearly extend $\zeta_\mathcal{F}$, $\zeta_\mathcal{F}^{\star}$, $\tilde\zeta_\mathcal{F}$, and $\tilde\zeta_\mathcal{F}^{\star}$ to $\mathcal{I}$. 
\begin{thm}[Hirose-Imatomi-Murahara-Saito {\cite[Lemma 2.5]{HIMS18}}] \label{ohnoFS}
 For a nonempty index $\boldsymbol{k}$ and a nonnegative integer $m$, we have
 \[ 
  \tilde\zeta_\mathcal{F}^\star (\boldsymbol{k} \sha (\{1\}^m) )
  =\sum_{ \left|\boldsymbol{e}\right|=m }
  \tilde\zeta_\mathcal{F}^\star (\boldsymbol{k}\oplus\boldsymbol{e}). 
 \]
\end{thm}

%%%%%%%%%%%%%%%%%%%%%%%%
\subsection{Calculation of $\sum_{\boldsymbol{e}} \tilde\zeta_\mathcal{F}^{\star} ((k_1,k_2,k_3) \oplus \boldsymbol{e})$}
We use Hoffman's algebraic setup with a slightly different convention (see \cite{Hof97}).
Let $\mathfrak{H}=\mathbb{Q} \left\langle x,y \right\rangle$ be the noncommutative polynomial ring in two indeterminates $x$ and $y$. 
We define a $\mathbb{Q}$-linear map $p\colon y\mathfrak{H}y\to \mathcal{I}$ by $p(yx^{k_1-1}y\cdots x^{k_r-1}y)=(k_1,\dots,k_r)$. 
For positive integers $l_1,l_2,l_3$, and a nonegative integer $m$, we define a polynomial $P_{m}(l_{1},l_{2},l_{3})$ in $\mathfrak{H}$ by
\begin{align*}
 P_{m}(l_{1},l_{2},l_{3}) 
  &=(-1)^{m} \sum_{ \substack{ e_1+e_2+e_3=m \\ e_1,e_2,e_3\ge0 }}
  y^{l_{1}+e_1}xy^{l_{2}+e_2-1}xy^{l_{3}+e_3}. 
\end{align*}
\begin{lem} \label{lemA}
 For positive integers $k_1,k_2,k_3$, we have
 \begin{align*} 
  &\sum_{\boldsymbol{e}} 
   \tilde\zeta_\mathcal{F}^{\star} ((k_1,k_2,k_3) \oplus \boldsymbol{e}) \\
  &=F_{k_1,1}(X) F_{k_3,1}(X)
   -\sum_{\substack{i+j\le k_2+1 \\ i,j \ge2}}
   \sum_{\substack{ n_1 \ge k_1+i-1 \\ n_3 \ge k_3+j-1 }}
    (-1)^{i} \binom{n_1}{i-1} 
    \tilde{\mathfrak{Z}}_\mathcal{F}(n_1)
    \times
    (-1)^{j} \binom{n_3}{j-1}
    \tilde{\mathfrak{Z}}_\mathcal{F}(n_3) \\
  &\quad +(-1)^{k_1+k_2+k_3} \sum_{m\ge0} 
   \tilde{\zeta}_{\mathcal{F}} (p(P_m(k_3,k_2,k_1))).
 \end{align*}
% We note that when $k_2=1,2$, the right-hand side is equal to 
% \[
%  F_{k_1,1}(X) F_{k_3,1}(X)
%  +(-1)^{k_1+k_2+k_3} \sum_{m\ge0} \tilde{\zeta}_{\mathcal{F}} (p(P_m(k_3,k_2,k_1))).
% \]
\end{lem}
\begin{proof}
 We prove this lemma only for $k_2\ge2$. 
 The case $k_2=1$ can be proved similarly. 
 Put $\boldsymbol{e}=(e_1,e_2,e_3)$.
 By Theorem \ref{dualFS}, we have
 \begin{align*}
  &\sum_{\boldsymbol{e}} \tilde\zeta_\mathcal{F}^\star ((k_1,k_2,k_3) \oplus \boldsymbol{e}) \\
  &=-\sum_{\boldsymbol{e}} \tilde\zeta_\mathcal{F}^\star (((k_1,k_2,k_3) \oplus \boldsymbol{e})^\vee) \\
  &=-\sum_{\boldsymbol{e}} \tilde\zeta_\mathcal{F}^\star (\{1\}^{k_1+e_1-1},2,\{1\}^{k_2+e_2-2},2,\{1\}^{k_3+e_3-1}).
 \end{align*}
 By Propositions \ref{symsumF}, \ref{antipode}, and \ref{112}, we have
 \begin{align*}
  &-\sum_{\boldsymbol{e}} 
   \tilde\zeta_\mathcal{F}^\star (\{1\}^{k_1+e_1-1},2,\{1\}^{k_2+e_2-2},2,\{1\}^{k_3+e_3-1}) \\
  &=\sum_{\boldsymbol{e}} 
   \Biggl( 
    \sum_{ \substack{ i+j=k_2+e_2+2 \\ i,j\ge2 } }
    (-1)^{k_3+e_3+j} 
    \tilde\zeta_\mathcal{F}^\star (\{1\}^{k_1+e_1-1},2,\{1\}^{i-2}) 
    \tilde{\zeta}_{\mathcal{F}} (\{1\}^{k_3+e_3-1},2,\{1\}^{j-2}) \\
   &\qquad\quad +(-1)^{k_1+k_2+k_3+e_1+e_2+e_3} 
    \tilde{\zeta}_{\mathcal{F}} (\{1\}^{k_3+e_3-1},2,\{1\}^{k_2+e_2-2},2,\{1\}^{k_1+e_1-1}) 
   \Biggr) \\
  &=\sum_{\boldsymbol{e}} 
    \sum_{ \substack{ i+j=k_2+e_2+2 \\ i,j\ge2 } }
     (-1)^{k_1+e_1+j} \binom{k_1+e_1+i-1}{k_1+e_1} 
     \tilde{\mathfrak{Z}}_{\mathcal{F}}(k_1+e_1+i-1) \\
    &\qquad\qquad\qquad\qquad\qquad\quad\,\,
     \times \binom{k_3+e_3+j-1}{k_3+e_3} \tilde{\mathfrak{Z}}_{\mathcal{F}}(k_3+e_3+j-1) \\
    &\quad +(-1)^{k_1+k_2+k_3} \sum_{m\ge0} 
     \tilde{\zeta}_{\mathcal{F}} (p(P_m(k_3,k_2,k_1))) \\
  &=\sum_{ \substack{ i+j\ge k_2+2 \\ i,j\ge2 } }
   \sum_{\substack{ n_1\ge k_1+i-1 \\ n_3\ge k_3+j-1 }} 
   (-1)^{n_1+i+j+1} 
   \binom{n_1}{i-1} \tilde{\mathfrak{Z}}_{\mathcal{F}}(n_1) 
   \times \binom{n_3}{j-1} \tilde{\mathfrak{Z}}_{\mathcal{F}}(n_3) \\
  &\quad +(-1)^{k_1+k_2+k_3} \sum_{m\ge0} 
   \tilde{\zeta}_{\mathcal{F}} (p(P_m(k_3,k_2,k_1))).
 \end{align*}
 Since $(-1)^{a+1} \mathfrak{Z}_\mathcal{F} (a)=\mathfrak{Z}_\mathcal{F} (a)$, we have
 \begin{align*} 
  \begin{split}
  &\sum_{\boldsymbol{e}} \tilde\zeta_\mathcal{F}^{\star} ((k_1,k_2,k_3) \oplus \boldsymbol{e}) \\
  &=\sum_{ \substack{ i+j\ge k_2+2 \\ i,j\ge2 } }
   \sum_{\substack{ n_1\ge k_1+i-1 \\ n_3\ge k_3+j -1}} 
   (-1)^{i} \binom{n_1}{i-1}\tilde{\mathfrak{Z}}_{\mathcal{F}}(n_1)
   \times (-1)^{j} 
   \binom{n_3}{j-1}\tilde{\mathfrak{Z}}_{\mathcal{F}}(n_3) \\
  &\quad +(-1)^{k_1+k_2+k_3} \sum_{m\ge0} 
   \tilde{\zeta}_{\mathcal{F}} (p(P_m(k_3,k_2,k_1))) \\
  &=\Biggl( \sum_{i,j\ge2}-\sum_{\substack{ i+j\le k_2+1 \\ i,j\ge2 }} \Biggr)
   \sum_{\substack{ n_1\ge k_1+i-1 \\ n_3\ge k_3+j -1}} 
   (-1)^{i} \binom{n_1}{i-1}\tilde{\mathfrak{Z}}_{\mathcal{F}}(n_1)
   \times 
   (-1)^{j} \binom{n_3}{j-1}\tilde{\mathfrak{Z}}_{\mathcal{F}}(n_3) \\
  &\quad +(-1)^{k_1+k_2+k_3} \sum_{m\ge0} 
   \tilde{\zeta}_{\mathcal{F}} (p(P_m(k_3,k_2,k_1))).
  \end{split}
 \end{align*}
 Here, we note that
 \begin{align*} 
  &\sum_{ i,j\ge2 }
   \sum_{\substack{ n_1\ge k_1+i-1 \\ n_3\ge k_3+j -1}} 
   (-1)^{i} \binom{n_1}{i-1}\tilde{\mathfrak{Z}}_{\mathcal{F}}(n_1)
   \times (-1)^{j} \binom{n_3}{j-1}\tilde{\mathfrak{Z}}_{\mathcal{F}}(n_3) \\
  &=\sum_{\substack{n_1 \ge k_1+1 \\ n_3 \ge k_3+1 }}
   \sum_{\substack{ 2\le i \le n_1-k_1+1 \\ 2\le j \le n_3-k_3+1 }}
   (-1)^{i} \binom{n_1}{i-1}\tilde{\mathfrak{Z}}_{\mathcal{F}}(n_1)
   \times (-1)^{j} \binom{n_3}{j-1}\tilde{\mathfrak{Z}}_{\mathcal{F}}(n_3) \\
  &=\sum_{\substack{n_1 \ge k_1+1 \\ n_3 \ge k_3+1 }}
   \sum_{\substack{ 2\le i \le n_1-k_1+1 \\ 2\le j \le n_3-k_3+1 }} 
   \left( (-1)^{i} \left( \binom{n_1-1}{i-1}+\binom{n_1-1}{i-2} \right) 
   \tilde{\mathfrak{Z}}_{\mathcal{F}}(n_1) \right) \\
   &\quad\qquad\qquad\qquad\,\,\,\,\,\,\, \times
   \left( (-1)^{j} \left( \binom{n_3-1}{j-1}+\binom{n_3-1}{j-2} \right) 
   \tilde{\mathfrak{Z}}_{\mathcal{F}}(n_3) \right) \\
  &=\sum_{\substack{n_1 \ge k_1+1 \\ n_3 \ge k_3+1 }}
  \left(
   (-1)^{k_1} \binom{n_1-1}{k_1-1}+1
  \right)
  \mathfrak{\tilde{Z}}_{\mathcal{F}}(n_1) 
  \times
  \left(
   (-1)^{k_3} \binom{n_3-1}{k_3-1}+1
  \right)
  \tilde{\mathfrak{Z}}_{\mathcal{F}}(n_3) \\
  &=F_{k_1,1}(X) F_{k_3,1}(X).
 \end{align*}
 Thus we get the result. 
\end{proof}

%%%%%%%%%%%%%%%%%%%%%%%%
\subsection{Calculation of 
 $\sum_{\boldsymbol{e}} (-1)^{\left|\boldsymbol{e}\right|} \tilde\zeta_\mathcal{F}^{\star} ((k_1,k_2,k_3)^\vee \oplus \boldsymbol{e})$}
For positive integers $l_1,l_2,l_3$, and a nonegative integer $m$, we define a polynomial $Q_{m}(l_{1},l_{2},l_{3})$ in $\mathfrak{H}$ by
\begin{align*}
 Q_{m}(l_{1},l_{2},l_{3}) 
 &=\sum_{ \substack{ e_1+e_2+e_3=m \\ e_1,e_2,e_3\ge0 }}
  \binom{l_{1}+e_1-1}{e_1} \binom{l_{2}+e_2-2}{e_2} \binom{l_{3}+e_3-1}{e_3} \\
  &\qquad\qquad\quad\,\,\,\, \times y^{l_{1}+e_1}xy^{l_{2}+e_2-1}xy^{l_{3}+e_3}. 
\end{align*}
Here, when $l_2=1$, we understand 
\begin{align*}
 \binom{l_2+e_2-2}{e_2}=\binom{e_2-1}{e_2}= 
 \begin{cases}
  1 \quad(e_2=0), \\ 
  0 \quad(e_2\ge1).
 \end{cases}
\end{align*}
\begin{lem} \label{lemB}
 For positive integers $k_1,k_2,k_3$, we have
 \begin{align*}
  &\sum_{\boldsymbol{e}} 
    (-1)^{\left|\boldsymbol{e}\right|} \zeta_\mathcal{F}^{\star} ((k_1,k_2,k_3)^{\vee}\oplus\boldsymbol{e}) \\
  &=-\sum_{\substack{ i+j=k_2+2 \\ i,j \ge2 }} F_{k_1,i-1}(X) F_{k_3,j-1}(X)
   -(-1)^{k_1+k_2+k_3} \sum_{m\ge0} 
   \tilde{\zeta}_{\mathcal{F}} (p(Q_m(k_3,k_2,k_1))). 
 \end{align*} 
%  When $k_2=1$ in the above equality, we understand the right-hand side as 
% \[
% -(-1)^{k_1+k_2+k_3} \sum_{m\ge0} 
%   \tilde{\zeta}_{\mathcal{F}} (p(Q_m(k_3,k_2,k_1))).
% \]
\end{lem}
\begin{proof}
 We prove this lemma only for $k_2\ge2$. 
 The case $k_2=1$ can be proved similarly. 
 By Theorem \ref{ohnoFS}, we have
 \begin{align*}
  &\sum_{\boldsymbol{e}} 
   (-1)^{\left|\boldsymbol{e}\right|} \tilde\zeta_\mathcal{F}^\star ((k_1,k_2,k_3)^{\vee}\oplus\boldsymbol{e}) \\
  &=\sum_{m\ge0}  
   (-1)^m \tilde\zeta_\mathcal{F}^\star ((k_1,k_2,k_3)^\vee \sha (\{1\}^m)) \\
  &=\sum_{m\ge0}  
   (-1)^m \tilde\zeta_\mathcal{F}^\star ((\{1\}^{k_1-1},2,\{1\}^{k_2-2},2,\{1\}^{k_3-1}) \sha (\{1\}^m)). 
 \end{align*}
 By Propositions \ref{symsumF} and \ref{antipode}, we have
 \begin{align*} \label{eq21}
  \begin{split} 
  &\sum_{m\ge0}  
   (-1)^m \tilde\zeta_\mathcal{F}^\star ((\{1\}^{k_1-1},2,\{1\}^{k_2-2},2,\{1\}^{k_3-1}) \sha (\{1\}^m)) \\
  &=-\sum_{\substack{ m_1\ge0 \\ m_3\ge0 }} 
    \sum_{\substack{ i+j=k_2+2 \\ i,j\ge2 }}
    (-1)^{k_3+m_1+j} 
    \tilde\zeta_\mathcal{F}^\star ((\{1\}^{k_1-1},2,\{1\}^{i-2}) \sha (\{1\}^{m_1})) \\
     &\qquad\qquad\qquad\qquad\qquad\quad\,\,\,\,
      \times \tilde{\zeta}_{\mathcal{F}} ((\{1\}^{k_3-1},2,\{1\}^{j-2}) \sha (\{1\}^{m_3})) \\
    &\quad-(-1)^{k_1+k_2+k_3}
     \sum_{m\ge0} 
     \tilde{\zeta}_{\mathcal{F}} ((\{1\}^{k_3-1},2,\{1\}^{k_2-2},2,\{1\}^{k_1-1}) \sha (\{1\}^m)). 
  \end{split}
 \end{align*}
 By Proposition \ref{112}, Corollary \ref{112no2}, and Theorems \ref{sumF} and \ref{ohnoFS}, we have
 \begin{align*} 
  &\sum_{\substack{ m_1\ge0 \\ m_3\ge0 }} 
    \sum_{\substack{ i+j=k_2+2 \\ i,j\ge2 }}
    (-1)^{k_3+m_1+j} 
    \tilde\zeta_\mathcal{F}^\star ((\{1\}^{k_1-1},2,\{1\}^{i-2}) \sha (\{1\}^{m_1})) \\
     &\qquad\qquad\qquad\qquad\quad\,\,\,
      \times \tilde{\zeta}_{\mathcal{F}} ((\{1\}^{k_3-1},2,\{1\}^{j-2}) \sha (\{1\}^{m_3})) \\
  &=(-1)^{k_1+k_2+k_3} \sum_{\substack{ m_1\ge0 \\ m_3\ge0 }} 
    \sum_{\substack{ i+j=k_2+2 \\ i,j\ge2 }} 
    \tilde\zeta_\mathcal{F}^\star ((\{1\}^{k_1-1},2,\{1\}^{i-2}) \sha (\{1\}^{m_1})) \\
     &\qquad\qquad\qquad\qquad\qquad\quad \times 
      \tilde{\zeta}_{\mathcal{F}}^\star ((\{1\}^{k_3-1},2,\{1\}^{j-2}) \sha (\{1\}^{m_3})) \\
  &=(-1)^{k_1+k_2+k_3} \sum_{\substack{ m_1\ge0 \\ m_3\ge0 }} 
      \sum_{\substack{ i+j=k_2+2 \\ i,j\ge2 }}
      \sum_{\left|\boldsymbol{e}_1\right|=m_1}
      \tilde\zeta_\mathcal{F}^\star ((\{1\}^{k_1-1},2,\{1\}^{i-2}) \oplus \boldsymbol{e}_1) \\
     &\qquad\qquad\qquad\qquad\qquad\quad\,\, \times 
      \sum_{\left|\boldsymbol{e}_3\right|=m_3}
      \tilde\zeta_\mathcal{F}^\star ((\{1\}^{k_3-1},2,\{1\}^{j-2}) \oplus \boldsymbol{e}_3) \\
   &=\sum_{\substack{ i+j=k_2+2 \\ i,j \ge2 }} 
     F_{k_1,i-1}(X) F_{k_3,j-1}(X).
 \end{align*}
 Since
 \[
  (\{1\}^{k_3-1},2,\{1\}^{k_2-2},2,\{1\}^{k_1-1}) \sha (\{1\}^m)
  =p(Q_m(k_3,k_2,k_1)),
 \]
 we have the desired result. 
\end{proof}

%%%%%%%%%%%%%%%%%%%%%%%%
\subsection{The equality $\tilde\zeta_\mathcal{F}(p(P_m(l_1,l_2,l_3)-Q_m(l_1,l_2,l_3)))=0$}
Recall $\mathfrak{H}=\mathbb{Q} \left\langle x,y \right\rangle$.
We define the shuffle product as the $\mathbb{Q}$-bilinear product $\sh:\mathfrak{H}\times\mathfrak{H}\to\mathfrak{H}$ given by  
\begin{align*}
 1\sh w&=w\sh 1=w, \\
 wu\sh w'u'&=(w\sh w'u')u+(wu\sh w')u',
\end{align*}
where $w,w'\in\mathfrak{H}$ and $u,u'\in\{x,y\}$. 
For $u_1,\dots,u_r\in\{x,y\}$, let $\overline{u_1\cdots u_r}=u_r\cdots u_1$. 
We denote by $\left|w\right|$ the degree of a word $w$, e.g., $\left|yx\right|=2$.
\begin{thm}[Kaneko-Zagier \cite{KZ19}] \label{shF}
 For words $w\in y\mathfrak{H}$ and $w'\in\mathbb{Q}\oplus y\mathfrak{H}$, we have
 \begin{align*}
  \zeta_\mathcal{F} (p((w\sh w')y))
  =(-1)^{\left|w'\right|} \zeta_\mathcal{F} (p(wy\overline{w'})).
 \end{align*}
\end{thm}
\begin{lem} \label{shF2}
 For $w\in y\mathfrak{H}$ and $w'\in\mathbb{Q}\oplus y\mathfrak{H}$, we have
 \[
  \zeta_\mathcal{F} (p(wy \sh w'y))=0.
 \]
\end{lem} 
\begin{proof}
 We may assume that $w$ and $w'$ are words. 
 By Theorem \ref{shF}, we have  
 \begin{align*}
  \zeta_\mathcal{F} (p(wy\sh w'y)) 
  &=\zeta_\mathcal{F} (p( (w\sh w'y)y + (wy\sh w')y) ) \\
  &=(-1)^{\left|w'y\right|} \zeta_\mathcal{F} (p( wy^2\overline{w'} ) ) + (-1)^{\left|w'\right|} \zeta_\mathcal{F} (p( wy^2\overline{w'} ) ) \\
  &=0. \qedhere
 \end{align*} 
\end{proof}
\begin{prop} \label{PQ}
 For positive integers $l_{1},l_{2},l_{3}$ and a nonnegative integer $m$, we have
 \[ 
  Q_{m}(l_{1},l_{2},l_{3}) 
  =\sum_{i=0}^{m} \sum_{ \substack{ e_1+e_2+e_3=i \\ e_1,e_2,e_3\ge0 }} (-1)^i y^{l_1+e_1}xy^{l_2+e_2-1}xy^{l_3+e_3} \sh y^{m-i}.
 \]
\end{prop}
\begin{proof}
 Fix nonnegative integers $a_1,a_2,a_3$ with $a_1+a_2+a_3=m$. 
 Then the coefficient of $y^{l_1+a_1}xy^{l_2+a_2-1}xy^{l_3+a_3}$ on the right-hand side is
 \begin{align*}
  &\sum_{j=0}^{a_1} (-1)^{a_1-j} \binom{l_1+a_1}{j} 
  \times \sum_{j=0}^{a_2} (-1)^{a_2-j} \binom{l_2+a_2-1}{j} 
  \times \sum_{j=0}^{a_3} (-1)^{a_3-j} \binom{l_3+a_3}{j} \\
  &=\sum_{j=0}^{a_1} (-1)^{a_1-j} \left( \binom{l_1+a_1-1}{j} +\binom{l_1+a_1-1}{j-1} \right) \\ 
   &\quad \times \sum_{j=0}^{a_2} (-1)^{a_2-j} \left( \binom{l_2+a_2-2}{j} +\binom{l_2+a_2-2}{j-1} \right) \\
   &\quad \times \sum_{j=0}^{a_3} (-1)^{a_3-j} \left( \binom{l_3+a_3-1}{j} +\binom{l_3+a_3-1}{j-1} \right) \\
  &=\binom{l_{1}+a_1-1}{a_1} \binom{l_{2}+a_2-2}{a_2} \binom{l_{3}+a_3-1}{a_3}.   
 \end{align*} 
 Here, we understand $\binom{n}{-1}=0$ for all integers $n$.
 This finishes the proof. 
\end{proof}
\begin{lem} \label{PQ2}
 For positive integers $l_{1},l_{2},l_{3}$ and a nonnegative integer $m$, we have
 \[ 
  \tilde\zeta_\mathcal{F} ( p(P_{m}(l_{1},l_{2},l_{3}) - Q_{m}(l_{1},l_{2},l_{3}) )) 
  =0.
 \]
\end{lem}
\begin{proof}
 By Proposition \ref{PQ}, we have
 \[ 
  P_{m}(l_{1},l_{2},l_{3}) - Q_{m}(l_{1},l_{2},l_{3}) 
  =-\sum_{i=0}^{m-1} \sum_{ \substack{ e_1+e_2+e_3=i \\ e_1,e_2,e_3\ge0 }} (-1)^i y^{l_1+e_1}xy^{l_2+e_2-1}xy^{l_3+e_3} \sh y^{m-i}.
 \]
 Thus, by Lemma \ref{shF2}, we have the desired result.
\end{proof}

%%%%%%%%%%%%%%%%%%%%%%%%
\subsection{Proof of Theorem \ref{main}}
Now we prove our main theorem.     
\begin{proof}[Proof of Theorem \ref{main}]
 By Lemmas \ref{lemA}, \ref{lemB}, and \ref{PQ2} we have
 \begin{align*}
  &\mathit{O}_\mathcal{F} (k_1,k_2,k_3) \\
  &=F_{k_1,1}(X) F_{k_3,1}(X)
   -\sum_{\substack{i+j\le k_2+1 \\ i,j \ge2}}
   \sum_{\substack{ n_1 \ge k_1+i-1 \\ n_3 \ge k_3+j-1 }}
    (-1)^{i} \binom{n_1}{i-1} 
    \tilde{\mathfrak{Z}}_\mathcal{F}(n_1)
    \times
    (-1)^{j} \binom{n_3}{j-1}
    \tilde{\mathfrak{Z}}_\mathcal{F}(n_3) \\ 
  &\quad -\sum_{\substack{ i+j=k_2+2 \\ i,j \ge2 }} F_{k_1,i-1}(X) F_{k_3,j-1}(X).
 \end{align*}
For positive integers $k$ and $s$, let
\[
 U_{k,s}
 =\sum_{n \ge s}(-1)^{k} \binom{n-1}{k-1} \mathfrak{\tilde{Z}}_\mathcal{F}(n).  
\]
 Since $(-1)^{s+1} \mathfrak{Z}_\mathcal{F} (s)=\mathfrak{Z}_\mathcal{F} (s)$, we note that
 \begin{align*}
  U_{k,s}-U_{k,s-1} 
  &=(-1)^{k-1} \binom{s-2}{k-1} \mathfrak{\tilde{Z}}_\mathcal{F}(s-1) 
  =(-1)^{s-k-1} \binom{s-2}{s-k-1} \mathfrak{\tilde{Z}}_\mathcal{F}(s-1) \\
  &=U_{s-k,s}-U_{s-k,s-1}. 
 \end{align*}
  Then we have
 \begin{align*}
  &\sum_{\substack{ i+j\le k_2+1 \\ i,j \ge2 }}
   \sum_{\substack{ n_1 \ge k_1+i-1 \\ n_3 \ge k_3+j-1 }}
    (-1)^{i} \binom{n_1}{i-1} \tilde{\mathfrak{Z}}_\mathcal{F}(n_1)
    \times
    (-1)^{j} \binom{n_3}{j-1} \tilde{\mathfrak{Z}}_\mathcal{F}(n_3) \\
  &=\sum_{\substack{i+j\le k_2+1 \\ i,j \ge2}}
   \sum_{\substack{ n_1 \ge k_1+i-1 \\ n_3 \ge k_3+j-1 }}
    (-1)^{i-1} \left( \binom{n_1-1}{i-2} +\binom{n_1-1}{i-1}  \right) \tilde{\mathfrak{Z}}_\mathcal{F}(n_1) \\
    &\quad\qquad\qquad\quad\quad\,\, \times
    (-1)^{j-1} \left( \binom{n_3-1}{j-2} +\binom{n_3-1}{j-1} \right) \tilde{\mathfrak{Z}}_\mathcal{F}(n_3) \\
  &=\sum_{\substack{ i+j\le k_2+1 \\ i,j \ge2 }}
    (U_{i-1,k_1+i-1}-U_{i,k_1+i-1})
    (U_{j-1,k_3+j-1}-U_{j,k_3+j-1}) \\
  &=\sum_{\substack{i+j\le k_2+1 \\ i,j \ge2 }}
    (U_{i-1,k_1+i-1}-U_{i,k_1+i}+U_{i,k_1+i}-U_{i,k_1+i-1}) \\
   &\quad\quad\quad\,\,\, \times
    (U_{j-1,k_3+j-1}-U_{j,k_3+j}+U_{j,k_3+j}-U_{j,k_3+j-1}) \\
   &=\sum_{\substack{ i+j\le k_2+1 \\ i,j \ge2 }}
     (U_{i-1,k_1+i-1}-U_{i,k_1+i}+U_{k_1,k_1+i}-U_{k_1,k_1+i-1}) \\
   &\quad\quad\quad\,\,\, \times
     (U_{j-1,k_3+j-1}-U_{j,k_3+j}+U_{k_3,k_3+j}-U_{k_3,k_3+j-1}) \\
  &=\sum_{\substack{ i+j\le k_2+1 \\ i,j \ge2 }} 
   (F_{k_1,i}(X)-F_{k_1,i-1}(X))(F_{k_3,j}(X)-F_{k_3,j-1}(X)).
 \end{align*}
 Thus we get
 \begin{align*}
  &\mathit{O}_\mathcal{F} (k_1,k_2,k_3) \\
  &=F_{k_1,1}(X) F_{k_3,1}(X)
   -\sum_{\substack{ i+j\le k_2+1 \\ i,j \ge2 }} (F_{k_1,i}(X)-F_{k_1,i-1}(X))(F_{k_3,j}(X)-F_{k_3,j-1}(X)) \\
   &\quad -\sum_{\substack{ i+j=k_2+2 \\ i,j \ge2 }} F_{k_1,i-1}(X) F_{k_3,j-1}(X). 
 \end{align*}
 When $k_2\ge2$, we note that
 \begin{align*}
  &\sum_{\substack{ i+j\le k_2+1 \\ i,j \ge2 }} (F_{k_1,i}(X)-F_{k_1,i-1}(X))(F_{k_3,j}(X)-F_{k_3,j-1}(X)) \\
  &=F_{k_1,1}(X) F_{k_3,1}(X)
   -\sum_{\substack{ i+j=k_2+2 \\ i,j \ge2 }} 
    F_{k_1,i-1}(X) F_{k_3,j-1}(X)
   +\sum_{\substack{i+j=k_2+1 \\ i,j\ge2}} 
    F_{k_1,i}(X) F_{k_3,j}(X). 
 \end{align*}
 Hence we find
 \begin{align*}
  \mathit{O}_\mathcal{F} (k_1,k_2,k_3) 
  =
  \begin{cases}
   \displaystyle{ 
    -\sum_{\substack{i+j=k_2+1 \\ i,j\ge2}} 
    F_{k_1,i}(X) F_{k_3,j}(X) \qquad (k_2\ge2), 
   } \\ 
    F_{k_1,1}(X) F_{k_3,1}(X) \qquad\qquad\qquad\, (k_2=1).
  \end{cases}
 \end{align*}
 This finishes the proof. 
\end{proof}

%%%%%%%%%%%%%%%%%%%%%%%%%%%%%%%%%%%%%%%%%%%%%%%%%%%%%%%%%%%%%%%%%%%%%%%%%%%%%%%%%%%%%%%%%%%%%%%%%%%%
\section*{Acknowledgements}
This work was supported by JSPS KAKENHI Grant Numbers JP18J00982, JP18K03243, and JP18K13392.

\end{document}